\documentclass[oneside,leqno,11pt]{paper}
\usepackage{amsmath,amssymb,latexsym}
\usepackage{pdfpages}
\usepackage[colorinlistoftodos]{todonotes}
\usepackage{cite}
\usepackage{multirow}
\usepackage{pgfplots}

\pgfplotsset{compat=1.7}

\usepackage{graphicx}
\usepackage[english]{babel}
\usepackage[utf8]{inputenc}
\usepackage{amsfonts}
\usepackage{hyperref}
\usepackage{enumerate}
\usepackage{bibentry}
\nobibliography*

\usepackage{graphicx}
\usepackage{float}
\usepackage{comment}
\usepackage{caption}
\usepackage{subcaption}

\usepackage{siunitx} 

\usepackage{amsmath}

\DeclareMathOperator{\arcsinh}{arcsinh}

\makeatletter
\def\ball{{I\kern -.35em B}}
\def\tto{\rightrightarrows}
\def\bx{\bar x}
\def\by{\bar y}

\def\nns{neighborhoods\,}

\def\gph{\mathop{\rm gph}\nolimits}

\def\dist{\mathop{\rm dist}\nolimits}

\newtheorem{proof}{Proof}

\newtheorem{definition}{Definition}[section]
\newtheorem{proposition}{Proposition}[section]
\newtheorem{theorem}{Theorem}[section]

\newtheorem{lemma}{Lemma}[section]
\newtheorem{remark}{Remark}[section]
\newtheorem{example1}{Example}[section]

\pgfdeclarepatternformonly[\LineSpace]{my north east lines}{\pgfqpoint{-1pt}{-1pt}}{\pgfqpoint{\LineSpace}{\LineSpace}}{\pgfqpoint{\LineSpace}{\LineSpace}}%
{
	\pgfsetcolor{\tikz@pattern@color}
	\pgfsetlinewidth{0.1pt}
	\pgfpathmoveto{\pgfqpoint{0pt}{0pt}}
	\pgfpathlineto{\pgfqpoint{\LineSpace + 0.1pt}{\LineSpace + 0.1pt}}
	\pgfusepath{stroke}
}

\pgfdeclarepatternformonly[\LineSpace]{my north west lines}{\pgfqpoint{-1pt}{-1pt}}{\pgfqpoint{\LineSpace}{\LineSpace}}{\pgfqpoint{\LineSpace}{\LineSpace}}%
{
	\pgfsetcolor{\tikz@pattern@color}
	\pgfsetlinewidth{0.1pt}
	\pgfpathmoveto{\pgfqpoint{0pt}{\LineSpace}}
	\pgfpathlineto{\pgfqpoint{\LineSpace + 0.1pt}{-0.1pt}}
	\pgfusepath{stroke}
}
\makeatother

\newdimen\LineSpace
\tikzset{
	line space/.code={\LineSpace=#1},
	line space=5pt
}
\begin{document}

\title{On Josephy-Halley method for generalized equations}
\author{Tomáš Roubal\thanks{Institute of Information Theory and Automation, Czech Academy of Sciences, Prague, Czech Republic, 	 roubal@utia.cas.cz, ORCID ID: 0000-0002-6137-1046}
	\and 
	    Jan Valdman\thanks{Institute of Information Theory and Automation, Czech Academy of Sciences, Prague, Czech Republic,	 jan.valdman@utia.cas.cz, ORCID ID: 0000-0002-6081-5362}
	    \thanks{Department of Mathematics, Faculty of Science, University of South Bohemia, Brani\v{s}ovsk\'a 31, \v{C}esk\'e Bud\v{e}jovice, Czech Republic}
	    }
\date{}
\maketitle

\begin{abstract}
We extend the classical third-order Halley iteration to the setting of generalized equations of the form
\[
0 \in f(x) + F(x),
\]
where \(f\colon X\longrightarrow Y\) is twice continuously Fréchet-differentiable on Banach spaces and \(F\colon X\tto Y\) is a set-valued mapping with closed graph.  Building on predictor–corrector framework, our scheme first solves a partially linearized inclusion to produce a predictor \(u_{k+1}\), then incorporates second-order information in a Halley-type corrector step to obtain \(x_{k+1}\).  Under metric regularity of the linearization at a reference solution and Hölder continuity of \(f''\), we prove that the iterates converge locally with order \(2+p\) (cubically when \(p=1\)).  Moreover, by constructing a suitable scalar majorant function we derive semilocal Kantorovich-type conditions guaranteeing well-definedness and R-cubic convergence from an explicit neighbourhood of the initial guess.  Numerical experiments—including one- and two-dimensional test problems confirm the theoretical convergence rates and illustrate the efficiency of the Josephy–Halley method compared to its Josephy–Newton counterpart.

\end{abstract}

\section{Introduction}

Halley method is a third-order iterative technique for solving nonlinear equations, offering faster convergence compared to classical Newton method. Originally introduced by Edmond Halley in 1694, this method refines the Newton approach by incorporating second-order derivative, which allows it to achieve cubic convergence under suitable conditions. The method is particularly useful when high accuracy is required, and it is applicable to both scalar equations and systems of nonlinear equations. In the context of Banach spaces and operator equations, Halley method has been extensively studied, e.g., in \cite{CQ1993, EH2007,LX2014, CR1985, HS1999, Yamamoto1988}, with convergence analyses extending from classical Kantorovich-type conditions to more modern approaches involving majorant functions \cite{CQ1993, SA2015,LX2014}, and logarithmic convexity indices\cite{HS1999}. Due to its cubic convergence rate, Halley method is a powerful tool in numerical analysis, especially for problems where evaluating higher-order derivatives is computationally feasible.
 Consider a single-valued mapping $f$ between Banach spaces $X$ and $Y$.
 We assume that $ f $ is continuously Fréchet differentiable and that there exists a solution $\bx \in X$ such that
\begin{eqnarray}
\label{eqEquality}
 f(\bx) = 0,
\end{eqnarray}
 with the Fréchet derivative $ f'(x) $.
 
 For a point \( x \) close to \( \bx \) and for a small perturbation \( h \in X \), the Taylor expansion of \( f \) at \( x \) is given by
$$
 f(x+h) = f(x) + f'(x)h + o(\|h\|).
$$
 To find a correction \( h \) such that \( f(x+h) = 0 \), we neglect the higher-order term and set
$$
 f(x) + f'(x)h \approx 0.
$$
 This leads to the Newton iteration
$$
 f(x_k) + f'(x_k)(x_{k+1}-x_k) = 0,\quad \text{for}\quad k=0,1,2,\dots
$$
 
 Halley method improves upon Newton method by incorporating second-order derivative information, which under suitable conditions leads to cubic convergence.
 
 Assume now that \( f \) is twice continuously Fréchet differentiable. For a perturbation \( h \in X \), the Taylor expansion up to second order is
 $$
 f(x+h) = f(x) + f'(x)h + \frac{1}{2} f''(x)(h)(h) + o(\|h\|^2),
 $$
 with the second Frechet derivative $f''(x)$.
 
 To find a correction \( h \) such that \( f(x+h) = 0 \), we neglect the higher-order term and set 
 \[
 f(x) + f'(x)h + \frac{1}{2} f''(x)(h)(h) \approx 0.
 \]
 A combination of the Newton iteration and this second-order correction leads to the Halley iteration scheme
 \begin{eqnarray}
 	\label{eqIterationHalleySingle1}
 	\left\lbrace	\begin{aligned}
 		&f(x_k)+f'(x_k)(u_{k+1}-x_k)=0,\\[1mm]
 		&f(x_k) + \Bigl(f'(x_k) + \tfrac{1}{2} f''(x_k)(u_{k+1}-x_k)\Bigr)(x_{k+1}-x_k)=0,
 	\end{aligned}
 	\right.
 	\quad \text{for}\quad k=0,1,2,\dots
 \end{eqnarray}
 Under certain regularity assumptions, for each \( k=0,1,2,\dots \), we obtain the iteration schema
 \begin{eqnarray}
 	\label{eqIterationHalleySingle2}
 	\left\lbrace	\begin{aligned}
 		u_{k+1}&=x_k-f'(x_k)^{-1} f(x_k),\\[1mm]
 		x_{k+1} &= x_k-\Bigl(f'(x_k) + \tfrac{1}{2} f''(x_k)(u_{k+1}-x_k)\Bigr)^{-1}f(x_k),
 	\end{aligned}
 	\right.
 \end{eqnarray}
 or equivalently,
 \begin{eqnarray*}
 	\label{eqIterationHalleySingle3}
 	\left\lbrace	\begin{aligned}
 		u_{k+1}&=x_k-f'(x_k)^{-1} f(x_k),\\[1mm]
 		x_{k+1} &= x_k-\Bigl(I_Y + \tfrac{1}{2} f'(x_k)^{-1}f''(x_k)(u_{k+1}-x_k)\Bigr)^{-1} f'(x_k)^{-1}f(x_k),
 	\end{aligned}
 	\right.
 \end{eqnarray*}
 where $I_Y$ is identity mapping in $Y$.
This approach has been extended for different types of approximations, e.g., \cite{EH2007}. Under certain assumptions, it provides third order of convergence to a solution of \eqref{eqEquality}. A sequence $(x_k)$ provides \emph{$p$-th order of convergence} to $\bx$ if $x_k \to \bx$ and
\begin{eqnarray*}
	\lim_{k\to\infty} \dfrac{\Vert x_{k+1}-\bx\Vert}{\Vert x_{k}-\bx\Vert^p}=L,
\end{eqnarray*}
where $L>0$.

In our approach, we extend the classical Halley method to the setting of \emph{generalized equations}. These problems consist in finding a point $x \in X$ such that 
\begin{eqnarray}
	 \label{eqGeneralizedEquation}
	  0 \in f(x) + F(x),
\end{eqnarray}
  where $f : X \longrightarrow Y$ is a single-valued and $F : X \rightrightarrows Y$ is a set-valued mapping, often representing a normal cone or a subdifferential in variational analysis.
  
   Generalized equations arise naturally in optimization, equilibrium problems, and variational inequalities, making the development of efficient and robust numerical methods for their solution a subject of significant interest. The focus of our work is to analyse the convergence properties of Halley method in this generalized framework.

 In 1979, N. H. Josephy introduced an iterative method in the spirit of Newton method for solving generalized equations \cite{Josephy1979}. The approach is based on a partial linearization of the generalized equation, following the ideas later formalized in \cite{Robinson1980}. The method is described by the following iterative scheme
 \begin{eqnarray} 
 	\label{eqIterationJosephyNewton}
 	 0 \in f(x_k) + f'(x_k)({x}_{k+1} - x_k) + F({x}_{k+1}),\quad \text{for}\quad k = 0, 1, 2, \dots 
 \end{eqnarray}
This method, along with its generalizations to the case of nonsmooth $f$, has been extensively studied and developed over the last several decades, e.g., \cite{Bonnans1994, DR2009, IKS2013,CDPVR2018}.

Our aim in this work is to analyse the convergence of an iterative scheme that blends the third‐order accuracy of Halley’s method \eqref{eqIterationHalleySingle1} with the generalized‐equation framework of Josephy–Newton \eqref{eqIterationJosephyNewton}. Let \(X\) and \(Y\) be Banach spaces, \(f\colon X\longrightarrow Y\) a twice  continuously Fréchet differentiable single‐valued mapping, and \(F\colon X\tto Y\) a set‐valued mapping with closed graph. Starting from an initial guess \(x_0\), each iteration proceeds by first solving the partially linearized generalized equation
\begin{eqnarray*}
  0 \in f(x_k)+ f'(x_k)\,(u_{k+1}-x_k)+F(u_{k+1}), \quad\text{for}\quad  k=0,1,2,\dots,
\end{eqnarray*}
to obtain a preliminary update \(u_{k+1}\). We then incorporate second‐order information by solving
\[
  0 \in f(x_k)+\Bigl(f'(x_k)+\tfrac12\,f''(x_k)(u_{k+1}-x_k)\Bigr)\,(x_{k+1}-x_k)+F(x_{k+1}),
\]
yielding the next iterate \(x_{k+1}\). In compact form, the full scheme reads
\begin{equation}
  \label{eqIterationHalley}
  \left\lbrace\begin{aligned}
    0&\in f(x_k)+ f'(x_k)\bigl(u_{k+1}-x_k\bigr)+F\bigl(u_{k+1}\bigr),\\
    0&\in f(x_k)+\Bigl(f'(x_k)+\tfrac12\,f''(x_k)\,(u_{k+1}-x_k)\Bigr)
             \bigl(x_{k+1}-x_k\bigr)+F\bigl(x_{k+1}\bigr),
  \end{aligned}\right.
  \quad\text{for}\quad k=0,1,2,\dots.
\end{equation}
Here \(f'(x_k)\) and \(f''(x_k)\) denote the first and second Fréchet derivatives of \(f\), respectively. We refer to this iteration as the {\em Josephy–Halley method}.

The first inclusion corresponds to the Josephy–Newton predictor step, while the second implements the Halley‐type corrector. Under suitable smoothness and metric regularity assumptions on the pair \((f,F)\), we shall prove that this combined method retains the cubic convergence rate characteristic of classical Halley’s method, even in the presence of the set-valued term \(F\).  

The paper is organized as follows.  In Section 2, we first fix our notation and recall fundamental concepts from variational analysis, including set-valued mappings, metric regularity, Fréchet derivatives, and generalized equations, which will be used throughout the paper. We also present key stability results for metric regularity under Lipschitz perturbations.

In Section~3, we turn to the local theory. Under a metric regularity assumption at a given solution of the generalized equation, we prove that the Josephy–Halley iteration converges cubically in a neighborhood of the root. In particular, we show that if the single‐valued part is twice continuously Fréchet differentiable with H\"older‐continuous second derivative, then the error satisfies
\[
  \|x_{k+1}-\bar x\|\le L\,\|x_k-\bar x\|^{2+p}
\]
for some \(p\in(0,1]\) and constant \(L>0\).

Section~4 is devoted to semilocal convergence in the spirit of  Kantorovich theorem. By introducing a suitable scalar majorant function \(h(t)\) and comparing the vector iterates to its roots, we derive explicit computable conditions on the initial guess under which the full Josephy–Halley scheme remains well‐defined and converges $R$‐cubically to a solution. We provide formulas for the radius of the convergence region and error bounds in terms of the majorant’s parameters.

Finally, in Section~5 we describe our MATLAB implementation of the Josephy–Halley method (using variable‐precision arithmetic for high accuracy) and present a suite of one‐ and two‐dimensional test problems. We tabulate observed convergence orders, compare iteration counts and runtimes against the Josephy–Newton method, and illustrate on a grid of starting points how the two schemes perform in practice.

\section{Preliminaries}
We denote Banach spaces by $(X, \Vert \cdot \Vert)$ and $(Y, \Vert \cdot \Vert)$. In these spaces, the closed ball and the open ball of radius $\delta > 0$ centered at a point $x \in X$ are defined by
\[
\ball_X[x, \delta] := \{u \in X : \Vert x - u \Vert \leq \delta \} \quad\text{and}\quad \ball_X(x, \delta) := \{u \in X : \Vert x - u \Vert < \delta \}.
\]

The distance from a point $x \in X$ to a  set $A \subseteq X$ is denoted by $\mathrm{dist}(x, A)$ and is defined as the shortest distance between $x$ and any point in $A$:
\[
\mathrm{dist}(x, A) := \inf_{u \in A} \Vert u - x \Vert.
\]
Note that $\mathrm{dist}(x, \emptyset)=\infty$.

We denote by $f'(x)$ the first Fréchet derivative of a mapping $f : X \longrightarrow Y$ at a point $x \in X$, and by $f''(x)$ the second Fréchet derivative of $f$ at $x \in X$. The first derivative $f'(x)$ is a bounded linear operator from $X$ to $Y$, while the second derivative $f''(x)$ is a bounded bilinear mapping from the Cartesian product $X \times X$ to $Y$. In particular, for any $h_1, h_2 \in X$, we have
\[
f''(x)(h_1) (h_2) \in Y.
\]

The set-valued mapping $F : X \rightrightarrows Y$ is given by its graph $\mathrm{gph}\, F\subset X\times Y$ such that $y\in F(x)$ if and only if $(x,y)\in \gph\,F$. The domain of $F$, denoted by $\mathrm{dom}\, F$, is the set of all points $x \in X$ for which $F(x) \neq \emptyset$. The inverse of $F$, denoted by $F^{-1}$, is defined as
\[
F^{-1}(y) := \{x \in X \mid y \in F(x) \}, \quad y \in Y.
\]

The $n$-dimensional Euclidean space is denoted by $\mathbb{R}^n$, with the Euclidean norm $\|x\|$ for $x \in \mathbb{R}^n$.



The regularity properties of mappings in the context of studying the convergence of the Josephy-Newton method were first employed by Bonnans in \cite{Bonnans1994}. We adopt this approach in our analysis. The concept of metric regularity has been extensively developed over the past decades and is now deeply embedded in the literature, e.g., the monograph by Dontchev and Rockafellar \cite{DR2014} for a comprehensive treatment and further references. Its definition follows.

\begin{definition}\label{defRegularity}Let $(X,\Vert \cdot\Vert)$ and $(Y,\Vert \cdot\Vert)$ be Banach spaces. 
	Let $F : X \rightrightarrows Y$ be a set-valued mapping and let $(\bar{x}, \bar{y}) \in \mathrm{gph}\, F$
	be a given point. We say that $F$ is \textit{metrically regular} around $(\bar{x}, \bar{y})$ if there is $\kappa \geq 0$ together with \nns $U$ of $\bar{x}$ and $V$ of $\bar{y}$ such that
		$$
		\mathrm{dist}(x, F^{-1}(y)) \leq \kappa\, \mathrm{dist}(y, F(x)) \quad \text{for each}\quad (x, y) \in U \times V.
		$$
\end{definition}
In our analysis, we need the following statement, which guarantees the qualitative stability of  metric regularity around the reference point under a single-valued Lipschitz continuous perturbation. This statement is taken from \cite[Theorem 2.3]{CPR2019}.
\begin{theorem}\label{thmSumStability} Let $(X, \| \cdot \|)$ and $(Y, \| \cdot \|)$ be Banach spaces, let $G : X \rightrightarrows Y$ be a set-valued mapping, and $(\bar{x}, \bar{y}) \in X \times Y$. Assume that there are positive constants $a$, $b$, and $\kappa$ such that the set $\gph \, G \cap (\ball_X[\bx, a] \times\ball_Y[\by, b])$ is closed in $X \times Y$ and $G$ is metrically regular on $\ball_X[\bx, a]$ for $\ball_Y[\by, b]$ with the constant $\kappa$. Let $\mu > 0$ be such that $\kappa \mu < 1$ and let $\kappa' > \kappa / (1 - \kappa \mu)$. Then for every positive $\alpha$ and $\beta$ such that
\begin{eqnarray}
	\label{eqConstantIneq}
2\kappa' \beta + \alpha \leq a \quad \text{and} \quad \mu (2\kappa' \beta + \alpha) + 2\beta \leq b
\end{eqnarray}
and for every mapping $g : X \longrightarrow Y$ satisfying
\[
\|g(\bar{x})\| \leq \beta \quad \text{and} \quad \|g(x) - g(x')\| \leq \mu \|x - u\| \text{ for every } x, x
u \in \ball_X [\bar{x}, {2\kappa' \beta + \alpha}],
\]
the mapping $g + G$ has the following property: for every $y,v \in \ball_Y[\by,\beta]$ and every $x \in (g + G)^{-1}(y) \cap \ball_X[\bar{x},\alpha]$ there exists a point $u \in \ball_X[\bx, 2\kappa' \beta + \alpha]$ such that
\[
v \in g(u) + G(u) \quad \text{and} \quad \|x - u\| \leq \kappa' \|y - v\|.
\]
\end{theorem}

We also need the following fixed point theorem for set-valued mappings, see \cite[Theorem 5E.2]{DR2014}.
\begin{theorem}\label{thmFixedPoint} Let $(X, \Vert\cdot\Vert)$ be a Banach space. Consider a set-valued mapping $\Phi : X \rightrightarrows X$ and a point $\bar{x} \in X$. Suppose that there exist scalars $a > 0$ and $\lambda \in (0,1)$ such that the set
$
\gph\,\Phi \cap (\ball_X[\bar{x},a] \times \ball_X[\bar{x},a])
$
is closed and
\begin{itemize}
    \item[\rm (a)] $\dist(\bx, \Phi(\bar{x})) < a(1 - \lambda)$;
    \item[\rm (b)] for each $z\in \Phi(x) \cap \ball_X[\bar{x}, a]$ we have
    $\dist(w, \Phi(v)) \leq \lambda \Vert x-v\Vert \quad \text{for all } x, v \in \ball_X[\bar{x},a]$.
\end{itemize}
Then $\Phi$ has a fixed point in $\ball_X[\bar{x},a]$; that is, there exists $x \in \ball_X[\bar{x}, a]$ such that $x \in \Phi(x)$.
\end{theorem}
\section{Local convergence of Josephy-Halley method}
In this section, we focus on the local convergence of the iterative scheme \eqref{eqIterationHalley}. We show that, under the assumption of metric regularity and H\"older continuity of the second derivative of $f$, the sequence generated by the iteration scheme converges to a solution of the generalized equation with the convergence rate $2 + p$. 

The mapping $f : X \longrightarrow Y$ is said to be H\"older continuous of order $p \in (0, 1]$ around a point $\bar{x} \in X$, if there exist constants $\ell > 0$ and $r > 0$ such that
\[
\Vert f(x) - f(u) \Vert \leq \ell \Vert x - u \Vert^p \quad \text{for all } x, u \in \ball_X[\bar{x}, r].
\]
Note that if $p = 1$, this condition reduces to Lipschitz continuity of $f$ around $\bar{x}$

\begin{theorem}\label{thmHalleyLocal}Let $(X,\Vert \cdot\Vert)$ and $(Y,\Vert \cdot\Vert)$ be Banach spaces. Consider a single-valued mapping $f:X\longrightarrow Y$, a set-valued mapping $F:X\tto Y$ and a point $\bx \in X$ satisfying \eqref{eqGeneralizedEquation}. Suppose that $f$ is twice continuously Fréchet differentiable around $\bx$ with H\"older continuous second derivative of order $p\in (0,1]$ around $\bx$ and $F$ has a closed graph. Assume that
the mapping $f(\bx)+f'(\bx)(\cdot-\bx)+F$ is metrically regular around $(\bx,0)$. Then there is $\gamma>0$ such that for each $x_0\in \ball_X[\bx, \gamma]$ the scheme \eqref{eqIterationHalley} generates sequences $( x_k)$ and  $( u_k)$  such that $(x_k)$ converges  to $\bx$. Moreover, there is $L>0$ such that
\begin{eqnarray}
	\Vert x_{k+1}-\bx\Vert \leq L \Vert x_k-\bx\Vert^{2+p}\quad\text{for each}\quad k\in \mathbb{N}_0
\label{cubic_convergence}
\end{eqnarray}
\end{theorem}

\begin{proof}
	At first, find $\kappa>0$, $a>0$, $b>0$, and $\ell>0$ such that $f(\bx)+f'(\bx)(\cdot-\bx)+F$ is metrically regular around $(\bx,0)$ with the constant $\kappa>0$ and \nns $\ball_X[\bx,a]$ and $\ball_Y[0,b]$, $f'$ is Lipschitz continuous on $\ball_X[\bx, a]$  with the constant $\ell$ and $f''$ is H\"older continuous on $\ball_X[\bx, a]$ of order $p$ and constant $\ell$.
	
	Moreover, for each $x\in \ball_X[\bx, a]$, see  \cite[Proposition 1]{Wells1973},	 we have 
	 \begin{eqnarray*}
	 	\begin{aligned}
	 	\Vert f(x)-f(\bx)-f'(\bx)(x-\bx)\Vert&\leq& \tfrac{\ell}{2}\Vert x-\bx\Vert^2,\\
	 	\Vert f'(x)-f'(\bx)-f''(\bx)(x-\bx)\Vert&\leq& \tfrac{\ell}{2}\Vert x-\bx\Vert^{1+p},\\
	 	\Vert f(x)-f(\bx)-f'(\bx)(x-\bx)-\tfrac{1}{2} f''(\bx)({x}-\bx)(x-\bx)\Vert &\leq&  \tfrac{\ell}{6}\Vert x-\bx\Vert^{2+p}.
	 	\end{aligned}
	 \end{eqnarray*}
	Applying the previous inequalities, for each $x\in \ball_X[\bx, a]$ and each $u\in X$  we get that 
	\begin{eqnarray*}
	&&	\Vert f(x)-f(\bx)-f'(x)(x-\bx)\Vert\leq \Vert f(x)-f(\bx)-f'(\bx)(x-\bx)\Vert+ \Vert f'(\bx)(x-\bx)-f'(x)(x-\bx)\Vert\\&\leq& \tfrac{\ell}{2}\Vert x-\bx\Vert^2+\ell\Vert x-\bx\Vert^2=\tfrac{3\ell}{2}\Vert x-\bx\Vert^2
	\end{eqnarray*}
	and
	\begin{eqnarray*}
		&&	\Vert f(x)-f(\bx)-\left(f'(x)+\tfrac{1}{2} f''(x)(u-x)\right)(x-\bx)\Vert\\ &\leq& 	\Vert f(x)-f(\bx)-f'(\bx)(x-\bx)-\tfrac{1}{2} f''(\bx)({x}-\bx)(x-\bx)\Vert\\ 
		&+&\Vert f'(\bx)(x-\bx)-f'(x)(x-\bx)+ \tfrac{1}{2} f''(\bx)(x-\bx)(x-\bx)-\tfrac{1}{2} f''(x)(u-x)(x-\bx)\Vert\\
		&\leq& \tfrac{\ell}{6} \Vert x-\bx\Vert^{2+p}+\Vert f'(\bx)(x-\bx)-f'(x)(x-\bx)+  f''(\bx)(x-\bx)(x-\bx)\Vert\\
		&+&\Vert-\tfrac{1}{2} f''(\bx)(x-\bx)(x-\bx)-\tfrac{1}{2} f''(x)(u-x)(x-\bx)\Vert
			\leq  \tfrac{2\ell}{3} \Vert x-\bx\Vert^{2+p}\\
			&+&\Vert-\tfrac{1}{2} f''(\bx)(x-\bx)(x-\bx)-\tfrac{1}{2} f''(x)(u-x)(x-\bx)\Vert\\
			&\leq& \tfrac{2\ell}{3} \Vert x-\bx\Vert^{2+p}+ \Vert-\tfrac{1}{2} f''(\bx)(x-\bx)(x-\bx)+\tfrac{1}{2} f''(x)(x-\bx)(x-\bx)\Vert+\tfrac{1}{2} \Vert f''(x)(u-\bx)(x-\bx)\Vert\\
			&\leq & \tfrac{7\ell}{6} \Vert x-\bx\Vert^{2+p}+\tfrac{\ell}{2}\Vert u-\bx\Vert \Vert x-\bx\Vert.
	\end{eqnarray*}
	
	 Fix positive numbers $\alpha$, $\beta$, $\kappa'$, and $\mu$ such that \eqref{eqConstantIneq}, $\kappa \mu < 1$,  $\kappa' > \kappa / (1 - \kappa \mu)$, and $\kappa'\mu <1$ hold.
	
	 For $u, v\in X$, define a mappings $g_{u,v}:{X}\longrightarrow Y$ given by
	\begin{eqnarray*}
		g_{u,v}(x):=f'(u)(x-\bx)+\tfrac{1}{2} f''(u)(v-u)(x-\bx)-f'(\bx)(x-\bx)\text{ for } x\in X.
	\end{eqnarray*}

	Let $\gamma := \min \left\lbrace 0.9,  \tfrac{\beta}{2\ell}, \tfrac{\mu}{2\ell}, \sqrt{\tfrac{2 \beta }{3\ell}}, \sqrt[p]{\tfrac{2}{3\ell\kappa'}}, \sqrt[1+p]{\tfrac{12}{\ell \kappa'(14+9\ell \kappa')}},\sqrt[2+p]{\tfrac{12\beta}{\ell(14+9\ell \kappa')}}  \right\rbrace$. 
	Fix any $ u, v\in \ball_{X}[\bx,\gamma]$. Then  
	\begin{eqnarray*}
		\Vert g_{u,v}(\bx)\Vert=0< \beta
	\end{eqnarray*}
	and for each $x, z\in\ball_X[\bx, 2\kappa' \beta + \alpha]$ we have
\begin{eqnarray*}
		&& \Vert g_{u,v}(x)-g_{u,v}(z)\Vert=\Vert  (f'(u)-f'(\bx))(x-z)+\tfrac{1}{2} f''(u)(v-u)(x-z)\Vert\\
		&\leq& (\ell \Vert u-\bx\Vert  +\tfrac{\ell}{2}\Vert v-u\Vert ) \Vert x-z\Vert\leq \mu\Vert x-z\Vert.
	\end{eqnarray*}
	Then, by Theorem \ref{thmSumStability}, with $\by:=0$, $g:=g_{u,v}$, and $G:=f(\bx)+f'(\bx)(\cdot-\bx)+F$, we get the claim that for each $y\in \ball_Y[0,\beta]$,  each $u\in \ball_{X}[\bx,2\kappa' \beta+\alpha]$, and each $v\in \ball_{X}[\bx,\gamma ]$ there is $x\in\ball_X[\bx,2\kappa'\beta+\alpha]$ such that
	\begin{eqnarray}
		y\in f(\bx)+f'(u)(x-\bx)+\tfrac{1}{2} f''(u)(v-u)(x'-\bx)+F(x)\quad\text{and}\quad\Vert x-\bx\Vert\leq \kappa'\Vert y\Vert.
	\end{eqnarray}

Fix any $x_0\in \ball_X[\bx, \gamma]$. Now, suppose that for some $k\in\mathbb{N}$ and for each $i= 1,2,\dots, k-1,k$ we have $x_i \in \ball_X[\bx,\gamma]$ and $u_i\in \ball_X[\bx, \gamma]$, generated by \eqref{eqIterationHalley}, such that
\begin{eqnarray}
	\label{eqEstimateError1}
	\Vert u_{i}-\bx\Vert\leq \tfrac{3\ell \kappa' }{2}  \Vert x_{i-1}-\bx\Vert^{1+p}\quad\text{and}\quad
	\Vert x_i-\bx\Vert\leq \kappa' \ell  \tfrac{14+9\ell \kappa'}{12}  \Vert x_{i-1}-\bx\Vert^{2+p}.%
\end{eqnarray}
We are showing that  there are $x_{k+1}\in\ball_{X}[\bx,\gamma]$ and $u_{k+1}\in \ball_{X}[\bx,\gamma]$, generated by \eqref{eqIterationHalley}, such that the   \eqref{eqEstimateError1} hold for $k+1$ instead of $i$.

By the definition of $\gamma$, we have $-(f(x_k)-f(\bx)-f'(x_k)(x_k-\bx))\in \ball_Y[0,\beta]$. For $u:=v:=x_k$, by the claim, with $y:=-(f(x_k)-f(\bx)-f'(x_k)(x_k-\bx))$, there is $u_{k+1}\in \ball_X[\bx,2\kappa'\beta+\alpha]$  satisfying the first inclusion in \eqref{eqIterationHalley}  such that 
\begin{eqnarray*}
 \Vert u_{k+1}-\bx\Vert\leq \kappa' \Vert y\Vert\leq   \tfrac{3\ell \kappa'}{2}\Vert x_k-\bx\Vert^2\leq \tfrac{3\ell \kappa'}{2}\Vert x_k-\bx\Vert^{1+p}\leq \gamma.
\end{eqnarray*}
By the definition of $\gamma$, we have $-(f(x_k)-f(\bx)-\left(f'(x_k)+\tfrac{1}{2} f''(x_k)(u_{k+1}-x_k)\right)(x_k-\bx))\in \ball_Y[0,\beta]$. For $u:=x_k$ and $v:=u_{k+1}$, by the claim, with $y:=-(f(\bx)-f(x_k)-\left(f'(x_k)+\tfrac{1}{2} f''(x_k)(u_{k+1}-x_k)\right)(\bx-x_k))$, there is $x_{k+1}\in \ball_X[\bx,2\kappa'\beta+\alpha]$ satisfying the second inclusion in \eqref{eqIterationHalley} such that 
\begin{eqnarray*}
&&\Vert x_{k+1}-\bx\Vert\leq \kappa'\Vert y\Vert\leq \tfrac{7\ell\kappa'}{6 .} \Vert x_k-\bx\Vert^{2+p}+\tfrac{\ell \kappa'}{2}\Vert u_{k+1}-\bx\Vert \Vert x_k-\bx\Vert
\leq \ell\kappa'   \tfrac{14+9\ell \kappa'}{12}\Vert x_k-\bx\Vert^{2+p}\leq\gamma.
\end{eqnarray*}
 By the induction, we have showed that \eqref{eqEstimateError1} holds for each $i\in \mathbb{N}_0$. Since $\gamma<1$, then $(x_k)$ converges to $\bx$.
%
%
\end{proof}
The result in Theorem 2.1 easily implies the single-valued version when the set-valued mapping is trivial, i.e.,  $F(x):=\lbrace 0\rbrace$. Remember when $f'(\bx)$ is subjective, then the mapping $X\ni x\longmapsto f'(\bx)x$ is metrically regular around $\bx$.
\begin{proposition}\label{thmHalleyLocal2} Let $(X,\Vert \cdot\Vert)$ and $(Y,\Vert \cdot\Vert)$ be Banach spaces. Consider a single-valued mapping $f:X\longrightarrow Y$, and a point $\bx \in X$ satisfying $f(\bx)=0$. Suppose that $f$ is twice continuously Fréchet differentiable around $\bx$ with H\"older continuous second derivative of the order $p\in (0,1]$ around $\bx$. Assume that the mapping $f'(\bx)$ is surjective. Then there is $\gamma>0$ such that for each $x_0\in \ball_{X}[\bx, \gamma]$ the scheme \eqref{eqIterationHalleySingle1} generates sequences $( x_k)$ and  $( u_k)$  such that $(x_k)$ converges  to $\bx$. Moreover, there is $L>0$ such that
	\begin{eqnarray*}
				\label{cubic_convergence2}
		\Vert x_{k+1}-\bx\Vert \leq L \Vert x_k-\bx\Vert^{2+p}\quad\text{for each}\quad k\in \mathbb{N}_0
	\end{eqnarray*}
\end{proposition}
\begin{proof}The statement follows from Theorem \ref{thmHalleyLocal} with $F:= \lbrace 0\rbrace$.
	\end{proof}

\section{Semilocal convergence of Josephy-Halley method}

In this section, we establish a semilocal convergence result for the iterative scheme \eqref{eqIterationHalley} by extending the classical Kantorovich framework. Our analysis leverages a modified majorant function technique, as introduced in \cite{LX2014}, to derive conditions under which the iteration converges to a solution of the generalized equation \eqref{eqGeneralizedEquation}.

To this end, for given positive constants $\kappa$, $\ell_1$, $\ell_2$, and $\eta$, we define the scalar function
\begin{eqnarray}
	\label{eqFunctionH}
h(t):=\frac{\kappa \ell_2}{6}\,t^3+\frac{\kappa \ell_1}{2}\,t^2-t+\eta,\quad t\in\mathbb{R}.
\end{eqnarray}
Furthermore, we introduce two sequences $(s_k)$ and $(t_k)$ by setting
\begin{equation}\label{eqSequences}
	t_0=0, s_0=0,\quad s_{k+1}=t_k-\frac{h(t_k)}{h'(t_k)},\quad
	t_{k+1}=t_k-\frac{h(t_k)}{h'(t_k)+\frac{1}{2}\,h''(t_k)(s_{k+1}-t_k)}, \text{ for }k=1,2,\ldots.
\end{equation}
 These sequences act as scalar majorants that control the propagation of the error in our iterative process. The convergence result \cite{LX2014} provides the foundation for proving that the iterative scheme \eqref{eqIterationHalley} converges semilocally to a solution of the generalized equation, with explicit estimates on the convergence rate and the radius of the convergence region.

\begin{lemma}\label{lemSequences} Let the  positive scalars $\kappa $, $\ell_1, \ell_2,$ and $\eta $ be such that  
	\begin{eqnarray*}
		\eta < \frac{2(\kappa \ell_1 + 2\sqrt{\kappa^2 \ell_1^2 + 2\kappa \ell_2})}{3(\kappa \ell_1 + \sqrt{\kappa^2 \ell_1^2 + 2\kappa \ell_2})^2}.
	\end{eqnarray*}
	Then
	\begin{enumerate}
		\item[\rm (i)]  $h$ has two positive roots, denoted by $\bar{t}$ and $\hat{t}$, satisfying $\bar{t} \leq \hat{t}$;
		\item[\rm (ii)] $-1<h'(t)<0$ for each $t\in (0, \bar{t})$;
		\item[\rm (iii)]   the sequences $(t_k)$ and $(s_k)$, generated by \eqref{eqSequences}, are increasing and converge to $\bar{t}$;
		\item[\rm (iv)] there are $M>0$ and $\alpha \in(0,1)$ such that   
		\begin{eqnarray*}
			\bar{t} - t_{k} \leq M \alpha^{3^k}\quad \text{for each}\quad k\in\mathbb{N}_0.
		\end{eqnarray*}
	\end{enumerate}	 
\end{lemma}
Before presenting the main result of this section, we first establish the following remark.
\begin{remark}
	Under the assumptions of Lemma \ref{lemSequences} and using \eqref{eqSequences}, for each $k\in\mathbb{N}_0$, we get 
	\begin{eqnarray*}
		h(t_k) &=& h(t_{k-1})+h'(t_{k-1})(t_k-t_{k-1})+\tfrac{1}{2} h''(t_{k-1})(t_k-t_{k-1})^2+\tfrac{1}{6} h'''(t_{k-1})(t_k-t_{k-1})^3\\
		&=& \tfrac{1}{2} h''(t_k)(s_{k}-t_{k-1})(t_k-t_{k-1}) +\tfrac{1}{2} h''(t_{k-1})(t_k-t_{k-1})^2+\tfrac{1}{6} h'''(t_{k-1})(t_k-t_{k-1})^3\\
		&\geq&\tfrac{\kappa\ell_1}{2}(s_k-t_k)(t_k-t_{k-1})+\tfrac{\kappa\ell_1}{2}(t_k-t_{k-1})^2+\tfrac{\kappa \ell_2}{6}(t_k-t_{k-1})^3.
	\end{eqnarray*}
\end{remark}
With the above remark and the preliminary estimates established, we are now in a position to state our main semilocal convergence result. The following theorem summarizes the conditions under which the iterative scheme \eqref{eqIterationHalley} converges to a solution of the generalized equation, and provides explicit error bounds on the convergence.

\begin{theorem}\label{thmKantorovich} Let $(X,\Vert \cdot\Vert)$ and $(Y,\Vert \cdot\Vert)$ be Banach spaces. Consider a set-valued mapping $F:X\tto Y$ with a closed graph and twice continuously differentiable mapping $f:X\longrightarrow Y$.	Let the positive scalars $a $, $b $, $\kappa$, $\ell_1, \ell_2,$ and $\eta$ and let a number $\bar{t}$ be the smallest positive root of function $h$ given by \eqref{eqFunctionH}. Suppose that the points $x_0 \in X$, $y_0 \in f(x_0) + F(x_0)$ are such that
	\begin{itemize}
\item[\rm (i)] $
   \kappa \Vert y_0\| < \eta<\dfrac{2(\kappa \ell_1+2\sqrt{\kappa^2 \ell_1^2+2\kappa \ell_2})}{3(\kappa \ell_1+\sqrt{\kappa^2 \ell_1^2+2\kappa \ell_2})^2},\quad$ $\quad\frac{3\ell_1}{2}\bar{t}^2+\Vert y_0\Vert<b,\quad$ and $\quad \bar{t}<a$;

	\item[\rm (ii)] the mapping
	$
	X\ni x \longmapsto G(x) := f(x_0) + f'(x_0)(x - x_0) + F(x)
	$
	is metrically regular around $(x_0, y_0)$ with constant $\kappa$ and \nns $\ball_X[x_0, a]$ and $\ball_Y[y_0,b]$;
	\item[\rm (iii)]  $\Vert f''(x)\Vert \leq \ell_1$ and
	$\Vert f''(x)-f''(u)\Vert\leq \ell_2\Vert x-u\Vert \quad\text{for each}\quad x,u\in \ball_X[x_0,a].$
\end{itemize}

Then  there exists a sequence $(x_k)$ generated by the iteration \eqref{eqIterationHalley} with the starting point $x_0$ which remains in $ \ball_X[x_0,\bar{t}]$ and converges to a solution $\bar{x} \in  \ball_X[x_0,\bar{t}]$ of \eqref{eqGeneralizedEquation}; moreover, the convergence rate is $R$-cubic, i.e., there are $\alpha\in (0,1)$ and $M>0$ such that
$$
\Vert x_k - \bar{x}\Vert \leq M \alpha^{3^k}\quad\text{and}\quad\dist(0,f(x_k)+F(x_k))  \leq M \alpha^{3^k} \quad \text{for every } k \in \mathbb{N}_0.
$$
\end{theorem}
\begin{proof} Consider sequences $(t_k)$ and  $(s_k)$ given by \eqref{eqSequences}.	Using Lemma \ref{lemSequences}, the sequences $(t_k)$ and $(s_k)$ are increasing and converge to $\bar{t}$ and  there are $M>0$ and $\alpha \in(0,1)$ such that   
\begin{eqnarray*}
	\bar{t} - t_{k} \leq M \alpha^{3^k}\quad \text{for each}\quad k\in\mathbb{N}_0.
\end{eqnarray*}
We are following proof of \cite[Theorem 3.4.]{CDPVR2018}.
	
	We are showing, by induction, that there are  sequences $(x_k)$ and $(u_k)$ in $\ball_X[x_0,\bar{t}]$ fulfilling \eqref{eqIterationHalley} with the starting point $x_0$ which for each  $k = 0, 1, 2, \ldots$ we have
	\begin{eqnarray}
		\label{eqEstimated}
	\|u_{k+1} - x_k\| \leq s_{k+1} - t_k\quad \text{and}\quad	\|x_{k+1} - x_{k}\| \leq t_{k+1} - t_{k}. 
	\end{eqnarray}
   Note that $\Vert x_k-x_0\Vert \leq t_k$.
   
For $k\in\mathbb{N}$, suppose that  \eqref{eqEstimated} holds for $k:=k-1$. For each $x\in\ball_{X}[x_k,\bar{t}-t_k]\subset \ball_{X}[x_0,\bar{t}]$ we have
\begin{eqnarray*}
	&&\Vert f(x_0) + f'(x_0)(x - x_0) - f(x_k) - f'(x_k)(x - x_k) \Vert \leq \Vert f(x)-f(x_0) - f'(x_0)(x - x_0)\Vert\\
	& +& \Vert f(x)- f(x_k) - f'(x_k)(x - x_k)\Vert\leq \tfrac{\ell_1}{2} \Vert x-x_0\Vert^2 + \tfrac{\ell_1}{2} \Vert x-x_k\Vert^2 \leq\ell_1 \bar{t}^2<b-\Vert y_0\Vert
\end{eqnarray*}
and
\begin{eqnarray*}
&&\Vert f(x_0) + f'(x_0)(x - x_0) - f(x_{k}) - f'(x_{k})(x - x_{k})-\tfrac{1}{2} f''(x_{k})(u_{k+1}-x_k)(x-x_{k})\Vert\\
&&\Vert f(x)-f(x_0) - f'(x_0)(x - x_0)\Vert
 + \Vert f(x)- f(x_k) - f'(x_k)(x - x_k)\Vert+\tfrac{\ell_1}{2} \Vert u_{k+1}-x_k\Vert \Vert x-x_{k}\Vert\\
  &\leq& \tfrac{\ell_1}{2} \Vert x-x_0\Vert^2 + \tfrac{\ell_1}{2} \Vert x-x_k\Vert^2 +\tfrac{\ell_1}{2} \Vert u_{k+1}-x_k\Vert \Vert x-x_{k}\Vert\leq\tfrac{3\ell_1}{2} \bar{t}^2<b-\Vert y_0\Vert.
\end{eqnarray*}

Define a set-valued mapping $\Phi_k^1: X\tto X$ given by
$$
\Phi_k^1(x) := G^{-1}\big(f(x_0) + f'(x_0)(x - x_0) - f(x_k) - f'(x_k)(x - x_k)\big)\quad\text{for}\quad x\in \ball_{X}[x_k,s_{k+1}-t_k].
$$

As the first step, applying Theorem \ref{thmFixedPoint}, with $\Phi:=\Phi_k^1, \bx:=x_k$ and $a:=s_{k+1} - t_k$, we demonstrate that the mapping $\Phi_k^1$ admits a fixed point $u_{k+1}$ in $\ball_{X}[x_k, s_{k+1} - t_k]$. This fixed point satisfies the first inclusion in \eqref{eqIterationHalley}. Then
\begin{eqnarray*}
	&&\dist(x_k,\Phi_k^1(x_k))=\dist(x_k, G^{-1}(f(x_0) + f'(x_0)(x - x_0) - f(x_k) ))\\
	&\leq& \kappa\dist(f(x_0) + f'(x_0)(x_k - x_0) - f(x_k) , G(x_k))\\
	&=&  \kappa\dist(0 ,f(x_k) +F(x_k))\leq\kappa(\tfrac{\ell_1}{2}\Vert u_k-x_{k-1}\Vert\Vert x_k-x_{k-1}\Vert+\tfrac{\ell_1}{2}\Vert x_k-x_{k-1}\Vert^2+\tfrac{ \ell_2}{6}\Vert x_k-x_{k-1}\Vert^3)\\
	&\leq &\tfrac{\kappa\ell_1}{2}(s_k-t_k)(t_k-t_{k-1})+\tfrac{\kappa\ell_1}{2}(t_k-t_{k-1})^2+\tfrac{\kappa \ell_2}{6}(t_k-t_{k-1})^3\leq h(t_k)
	=\tfrac{(1-\kappa \ell_1 t_k) h(t_k)}{1-\kappa \ell_1 t_k}\\ &\leq& - (1-\kappa \ell_1 t_k)\tfrac{ h(t_k)}{h'(t_k)}= (1-\kappa \ell_1 t_k)(s_{k+1}-t_k).
\end{eqnarray*}
Fix any $x, v \in\ball_X [x_k, s_{k+1} - t_k]$ and any $z \in \Phi_k^1(x) \cap \ball_X [x_k, s_{k+1} - t_k]$. Then
$$
f(x_0) + f'(x_0)(x - x_0) - f(x_{k}) - f'(x_{k})(x -x_{k})\in G(z)
$$
and so
\begin{eqnarray*}
	\dist(z, \Phi_k^1(v))&=&\dist(z, G^{-1}\left(f(x_0) + f'(x_0)(v - x_0) - f(x_{k}) - f'(x_{k})(v - x_{k})
	\right))\\	
	&\leq& \kappa \, \dist(f(x_0) + f'(x_0)(v - x_0) - f(x_{k}) - f'(x_{k})(v - x_{k}), G(z))\\
	&\leq& \kappa \Vert f'(x_0)(x-v)-f'(x_k)(x-v)\Vert
	\leq \kappa \ell_1 \|x_k-x_0\| \|x - v\|\\
	&\leq& \kappa \ell_1 t_k\|x - v\|.
\end{eqnarray*}
Since, by Lemma \ref{lemSequences}{\rm (ii)}, $0<\kappa \ell_1 t_k\leq \kappa \ell_1 t_k+\tfrac{\kappa \ell_2}{2}t_k^2=h'(t_k)+1 <1$,  there is $u_{k+1}$ satisfying the first inclusion in \eqref{eqIterationHalley} such that $\Vert u_{k+1}-x_k\Vert \leq s_{k+1}-t_k.$

Further, define a set-valued mapping $\Phi_k^2: X\tto X$ given by
$$
\Phi_k^2(x) := G^{-1}\big(f(x_0) + f'(x_0)(x - x_0) - f(x_{k}) - f'(x_{k})(x - x_{k})-\tfrac{1}{2} f''(x_{k})(u_{k+1}-x_k)(x-x_{k})\big)
$$
for $x\in\ball_{X}[x_k,t_{k+1}-t_k]$. 
As the second step, applying Theorem \ref{thmFixedPoint}, with $\Phi:=\Phi_k^2, \bx:=x_k$ and $a:=t_{k+1} - t_k$, we demonstrate that the mapping $\Phi_k^2$ admits a fixed point $x_{k+1}$ in $\ball_{X}[x_k, t_{k+1} - t_k]$. This fixed point satisfies the second inclusion in \eqref{eqIterationHalley}.

 Further,
\begin{eqnarray*}
	&&\dist(x_k,\Phi_k^2(x_k))=\dist(x_k, G^{-1}(f(x_0) + f'(x_0)(x - x_0) - f(x_k) ))\\
	&\leq& \kappa\dist(f(x_0) + f'(x_0)(x_k - x_0) - f(x_k) , G(x_k))\\
	&=&  \kappa\dist(0 ,f(x_k) +F(x_k))\leq\kappa(\tfrac{\ell_1}{2}\Vert u_k-x_{k-1}\Vert\Vert x_k-x_{k-1}\Vert+\tfrac{\ell_1}{2}\Vert x_k-x_{k-1}\Vert^2+\tfrac{ \ell_2}{6}\Vert x_k-x_{k-1}\Vert^3)\\
	&\leq &\tfrac{\kappa \ell_1}{2}(s_k-t_k)(t_k-t_{k-1})+\tfrac{\kappa \ell_1}{2}(t_k-t_{k-1})^2+\tfrac{ \kappa\ell_2}{6}(t_k-t_{k-1})^3\leq h(t_k)\\
		&=&\tfrac{(1-\kappa \ell_1 t_k-\tfrac{\kappa \ell_1}{2} (s_{k+1}-t_k)) h(t_k)}{1-\kappa \ell_1 t_k-\tfrac{\kappa \ell_1}{2} (s_{k+1}-t_k)}\leq- (1-\kappa \ell_1 t_k-\tfrac{\kappa \ell_1}{2} (s_{k+1}-t_k))\tfrac{ h(t_k)}{h'(t_k)+\tfrac{1}{2}h''(t_k)(s_{k+1}-t_k)}\\
		&=& (1-\kappa \ell_1 t_k-\tfrac{\kappa \ell_1}{2} (s_{k+1}-t_k))(t_{k+1}-t_k).
\end{eqnarray*}
Fix any $x, v \in\ball_X [x_k, t_{k+1} - t_k]$ and any $z \in \Phi_k^2(x) \cap \ball_X [x_k, t_{k+1} - t_k]$. Then
$$
f(x_0) + f'(x_0)(u - x_0) - f(x_{k}) - f'(x_{k})(x -x_{k})-\tfrac{1}{2}f''(x_k)(u_{k+1}-x_k)(x-x_k)\in G(z)
$$
and so
\begin{eqnarray*}
	&& \dist(z, G^{-1}\left(f(x_0) + f'(x_0)(v - x_0) - f(x_{k}) - f'(x_{k})(v - x_{k})-\tfrac{1}{2} f''(x_{k})(u_{k+1}-x_k)(v-x_{k})
	\right))\\	
	&\leq& \kappa \, \dist(f(x_0) + f'(x_0)(v - x_0) - f(x_{k}) - f'(x_{k})(v - x_{k})-\tfrac{1}{2} f''(x_{k})(u_{k+1}-x_k)(v-x_{k}), G(z))\\
	&\leq& \kappa \Vert f'(x_0)(x-v)-f'(x_k)(x-v)-\tfrac{1}{2} f''(x_{k})(u_{k+1}-x_k)(v-x)\Vert\\
	&\leq& (\kappa \ell_1  \|x_{k}-x_{0}\|+\tfrac{\kappa\ell_1}{2}\Vert u_{k+1}-x_k\Vert) \|u - v\|\leq  (\kappa \ell_1  t_k+\tfrac{\kappa\ell_1}{2}(s_{k+1}-t_k)) \|x - v\|.
\end{eqnarray*}
Since, by Lemma \ref{lemSequences}{\rm (ii)}, $\kappa \ell_1  t_k+\tfrac{\kappa\ell_1}{2} (s_{k+1}-t_k)<\kappa\ell_1 s_{k+1}<h'(s_{k+1})+1<1$, there exists $x_{k+1} \in X$ satisfying the second inclusion in \eqref{eqIterationHalley} such that $\Vert x_{k+1} - x_k \Vert \leq t_{k+1} - t_k$. Hence, we have shown that there are sequences $(u_k)$ and $(x_k)$ generated by \eqref{eqIterationHalley} such that \eqref{eqEstimated} holds for each $k \in \mathbb{N}_0$.

Since the sequence $(t_k)$ is convergent, the sequence $(x_k)$ is Cauchy. Therefore, there exists $\bar{x} \in X$ such that $(x_k)$ converges to $\bar{x}$. Then,  for each $k \in \mathbb{N}_0$, we have 
\begin{eqnarray*}
	 \Vert \bar{x} - x_k \Vert &=& \lim_{m \to \infty} \Vert x_m - x_k \Vert \\
	  &\leq& \lim_{m \to \infty} \left(\Vert x_{k+m} - x_{k+m-1} \Vert + \Vert x_{k+m-1} - x_{k+m-2} \Vert + \dots + \Vert x_{k+1} - x_k \Vert \right) \\
	   &=& \lim_{m \to \infty} (t_{k+m} - t_k) \leq \bar{t} - t_k \leq M \alpha^{3^k}.
\end{eqnarray*} Setting $k := 0$, we obtain $\Vert \bar{x} - x_0 \Vert \leq \bar{t}$. Therefore $\bar{x} \in \ball_{X}[x_0, \bar{t}]$. Since \( F \) has a closed graph and, for each \( k \in \mathbb{N}_0 \), we have
\[
\dist(0, f(x_k) + F(x_k)) \leq \left( 1 - \kappa \ell_1 t_k - \tfrac{\kappa \ell_1}{2}(s_{k+1} - t_k) \right)(t_{k+1} - t_k) \leq \bar{t} - t_k \leq M \alpha^{3^k},
\]
it follows that \( 0 \in f(\bar{x}) + F(\bar{x}) \).
\end{proof}
In the particular case where the set-valued mapping is trivial, i.e., $F(x)=\{0\}$, the analysis simplifies and we can derive explicit convergence results for the iterative scheme \eqref{eqIterationHalleySingle1}. The following proposition extends the Halley method to problems in which the linear mapping $f'(x_0)$ is not invertible but merely surjective.

\begin{proposition} Let $(X,\Vert \cdot\Vert)$ and $(Y,\Vert \cdot\Vert)$ be Banach spaces. Consider a twice continuously differentiable mapping $f:X\longrightarrow Y$.	Let the positive scalars $a $, $b $, $\kappa$, $\ell_1, \ell_2,$ and $\eta$ and let a number $\bar{t}$ be the smallest positive root of function $h$ given by \eqref{eqFunctionH}. Suppose that the point $x_0 \in X$ be such that
	\begin{itemize}
		\item[\rm (i)] $
		\kappa \Vert f(x_0)\| < \eta<\dfrac{2(\kappa \ell_1+2\sqrt{\kappa^2 \ell_1^2+2\kappa \ell_2})}{3(\kappa \ell_1+\sqrt{\kappa \ell_1+2\kappa \ell_2})^2}$\qquad and $\qquad\bar{t}<a$;
		
		\item[\rm (ii)] the mapping $X\ni x\longmapsto f'(x_0)x$ is metrically regular with the constant $\kappa$, i.e., for  each $y\in Y$ there is $x \in X$ such that $f'(x_0)x=y$ and
		\begin{eqnarray*}
			\Vert x\Vert\leq \kappa \Vert y\Vert;
		\end{eqnarray*}
		\item[\rm (iii)] 
	$\Vert f''(x)\Vert \leq \ell_1$ and	$\Vert f''(x)-f''(u)\Vert\leq \ell_2\Vert x-u\Vert \quad\text{for each}\quad x,u\in \ball_X[x_0,a]$.
	\end{itemize}
	
	Then  there exists a sequence $(x_k)$ generated by the iteration \eqref{eqIterationHalleySingle1} with the starting point $x_0$ which remains in $ \ball_X[x_0,\bar{t}]$ and converges to a solution $\bar{x} \in  \ball_X[x_0,\bar{t}]$ of \eqref{eqGeneralizedEquation}; moreover, the convergence rate is $R$-cubic, i.e, there are $\alpha\in (0,1)$ and $M>0$ such that
	$$
	\Vert x_k - \bar{x}\Vert \leq M \alpha^{3^k}\quad\text{and}\quad\Vert f(x_k)\Vert  \leq M \alpha^{3^k} \quad \text{for every } k \in \mathbb{N}_0.
	$$
\end{proposition}
\begin{proof}
    The statement follows from Theorem \ref{thmKantorovich}, with $F:= \lbrace 0\rbrace$, and arbitrary $b>0$ such that $\tfrac{3\ell_1\bar{t}^2}{2}+\Vert f(x_0)\Vert<b$.
\end{proof}

\section{Numerical experiments}
This section focuses on applying the iterative method described in  \eqref{eqIterationHalley} to various theoretical and practical problems. The analysis begins with a detailed exploration of its behavior in fundamental problems in mathematical theory, illustrating its convergence properties and effectiveness. 
By examining these use cases, we aim to demonstrate the robustness of the iteration scheme. The results presented herein validate the theoretical underpinnings of the iteration. Considered models are in the form \eqref{eqGeneralizedEquation}, which covers several problems arising in theory and engineering.   
We observe a sequence of absolute errors of iterations
$$e_{k} = x_{k}-\bx, \quad k\in \mathbb{N}_0 $$ and estimate the rate $r$ and the constant $L$ of 
\begin{eqnarray}
	\Vert e_{k+1} \Vert \approx L \Vert e_k \Vert ^r\quad\text{for each}\quad k\in \mathbb{N}_0.
\label{cubic_convergence_numerical}
\end{eqnarray}
According to the theoretical estimate \eqref{cubic_convergence} with $q:=1$, we should see $ r \approx 3$ 
for large values of $k$ to confirm the cubic rate of convergence. Practically, we can calculate approximations of $r, L$ from formulas
\begin{eqnarray}
r_{k+2} \approx 
\frac{
\log{\Vert e_{k+2}} \Vert - \log{\Vert e_{k+1} \Vert }
} 
{\!\!\!\!\!
\log{\Vert e_{k+1}} \Vert - \log{\Vert e_{k} \Vert }
}, 
\quad  
L_{k+2} \approx 
\frac{\,\,\,\!\!\!\!\!\!\!\!\!\!
\Vert e_{k+2} \Vert } 
{
{\,\,\,\Vert e_{k+1} \Vert }^{r_{k+2} }
}
\quad\text{for each}\quad k\in \mathbb{N}_0
\end{eqnarray}
that can be derived directly from \eqref{cubic_convergence_numerical}. 

Computations in MATLAB were enhanced using the Variable Precision Arithmetic (VPA) library to ensure high numerical precision. Specifically, we used the option "digits(400)" to perform calculations with up to 400 significant digits. This level of precision is essential for accurately observing cubic convergence, as the standard double-precision limit of 16 significant digits would be quickly exhausted during the iterative process. Furthermore, we had to develop our own MATLAB libraries to compute solutions for partially linearized generalized equations, since the built-in libraries do not support VPA. 

To better reading, we define two set-valued mappings $F_1:\mathbb{R}\tto\mathbb{R}$ and $F_2:\mathbb{R}\tto\mathbb{R}$ given by
\begin{eqnarray*}
    F_1(x):=\begin{cases}
    [0,\infty)&\quad\text{for}\quad x=0,\\
    0,&\quad\text{for}\quad x>0,\\
    \emptyset, &\quad\text{for}\quad x<0,
    \end{cases}\qquad\text{and}\qquad
      F_2(x):=\begin{cases}
    -1,&\quad\text{for}\quad x<0,\\
    [-1,1],& \quad\text{for}\quad x=0,\\
    1,& \quad\text{for}\quad x>0.
    \end{cases}
\end{eqnarray*}
The set-valued mapping \(F_1\) is the normal-cone mapping associated with the set \([0,\infty)\), and the set-valued mapping \(F_2\) is the Clarke subdifferential of the absolute-value function.

The following examples include the maximal monotone set-valued part (so is $F_1$ and $F_2$) and a strongly monotone and the infinitely continuously differentiable single-valued part; hence, the assumptions of Theorem \ref{thmHalleyLocal} are satisfied near the solution, see \cite[12.54 Proposition]{RW1998}.

\begin{example1}[one variable]\label{Ex1} 
 Consider a problem to find $x\in \mathbb{R}$ such that \eqref{eqGeneralizedEquation} holds,	where $f:\mathbb{R}\longrightarrow\mathbb{R}$ and $F:\mathbb{R}\tto \mathbb{R}$ are given. Specifically, we analyse the following two cases:
	\begin{enumerate}
		\item[\rm (i)] $f(x):=\sinh(x)-\frac{3}{8}$ and $F(x):=F_1(x)$ for $x\in\mathbb{R}$;
		\item[\rm (ii)] $f(x):=\sinh(x)+10$ and $F(x):=F_2(x)$ for $x\in\mathbb{R}$,
	\end{enumerate}
with known exact solutions 
\begin{enumerate}
		\item[\rm (i)] $\bar x=\arcsinh(3/8) \approx 0.3667246042301368$; 
		\item[\rm (ii)] $\bar x=-\arcsinh(9) \approx -2.893443985885871$.
	\end{enumerate}
 The cubic convergence is demonstrated in both cases in Tables \ref{tab:overall_example1i} and \ref{tab:overall_example1ii}.

\begin{table}[h!]
		\centering
		\begin{tabular}{c|c|c|c|c}
			$k$ & $x_k$ & $e_k$ & $r_k$ & $L_k$ \\ \hline
			 0 &  6.000000 &  5.63e+00 &  - &  -  \\ 
 1 &  4.007496 &  3.64e+00 &   - &  -  \\ 
 2 &  2.065099 &  1.70e+00 &  1.746923 &  0.177694  \\ 
 3 &  0.547668 &  1.81e-01 &  2.936618 &  0.038196  \\ 
 4 &  0.366054 &  6.71e-04 &  2.499630 &  0.048144  \\ 
 5 &  0.366725 &  4.11e-11 &  2.967402 &  0.107114  \\ 
 6 &  0.366725 &  9.40e-33 &  3.000034 &  0.135956  \\ 
 7 &  0.366725 &  1.13e-97 &  3.000000 &  0.135845  \\ 
 8 &  0.366725 &  1.95e-292 &  3.000000 &  0.135845  \\ 
		\end{tabular}
	\caption{Calculations (i) of Example \ref{Ex1}.}
	\label{tab:overall_example1i}
    \vspace{0.5cm}
		\centering
		\begin{tabular}{c|c|c|c|c}
			$k$ & $x_k$ & $e_k$ & $r_k$ & $L_k$ \\ \hline
			0 &  -10.000000 &  7.11e+00 &  - &  -  \\ 
 1 &  -8.003266 &  5.11e+00 &  - &  -  \\ 
 2 &  -6.027198 &  3.13e+00 &  1.482275 &  0.279263  \\ 
 3 &  -4.193713 &  1.30e+00 &  1.799143 &  0.166549  \\ 
 4 &  -3.049397 &  1.56e-01 &  2.410897 &  0.082808  \\ 
 5 &  -2.893750 &  3.06e-04 &  2.939091 &  0.072084  \\ 
 6 &  -2.893444 &  2.30e-12 &  3.000870 &  0.080853  \\ 
 7 &  -2.893444 &  9.83e-37 &  3.000001 &  0.080287  \\ 
 8 &  -2.893444 &  7.62e-110 &  3.000000 &  0.080285  \\
		\end{tabular}
	\caption{Calculations (ii) of Example \ref{Ex1}.}
	\label{tab:overall_example1ii}
\end{table}

\end{example1}

\begin{example1}[two variables]\label{Ex2} We define 
$f:\mathbb{R}^2\longrightarrow \mathbb{R}^2$ by
	\begin{eqnarray}
		\label{eqEx2}
		f(x_1,x_2):=(e^{x_1-x_2-p}-q_1, e^{x_1+x_2-p}-q_2),
	\end{eqnarray}
with parameters $p \in \mathbb{R}, q_1, q_2 \in \mathbb{R}^+$. The problem is to find $ (x_1,x_2)\in \mathbb{R}^2$ such that
\begin{eqnarray*}
    f(x_1,x_2)=0,
\end{eqnarray*}
has a known exact solution given by
\begin{eqnarray}
\label{eqSoltutionExample52}
(x_1,x_2)=
\left(\frac{1}{2}\log{(q_1 q_2)}+p, \frac{1}{2}\log{(q_2/q_1)} \right).
\end{eqnarray}
This setup is inspired by \cite{CR1985}, in which  
$f(x_1,x_2):=(e^{-x_1-x_2}-0.1, e^{-x_1+x_2}-0.1)$. Our modification ensures that the derivative of $f$  at each point is $p$-matrix (all main minors are positive); then it is guaranteed that $f$ is strongly monotone.

We consider two cases of the multifunction $F:\mathbb{R}^2\tto\mathbb{R}^2$:
\begin{itemize}
\item[\rm (i)] $F(x_1,x_2):=\lbrace (0,0)\rbrace;$
\item[\rm (ii)] $F(x_1,x_2):=F_2(x_1)\times F_2(x_2).$
\end{itemize}
In our simulations, we consider parameters $p=3, q_1=0.1,$ and $q_2=0.2$ for the case {\rm (i)} and $p=0, q_1=2.3,$ and $q_2=1$ for the case {\rm (ii)}

The cubic convergence of the Josephy-Halley method in case (i)
is demonstrated in Table \ref{tab:overall_example2}. It needs six iterations to converge to the exact solution 
$$(x_1,x_2) \approx (1.043988497285927, 0.346573590279973)$$
provided by the formula \eqref{eqSoltutionExample52}. The Josephy-Newton method requires eleven iterations to converge to the exact solution, see Figure \ref{fig:comparison_example2}. The total evaluation time of both methods is comparable due to the higher cost of the Halley method. 
\begin{table}[h!]
\centering
\begin{tabular}{c|c|c|c|c|c}
			$k$ & $x_k^1$ & $x_k^2$  & $e_k$ & $r_k$ & $L_k$ \\ \hline
 0 &  1.000000 &  -1.000000 &  1.35e+00 &  0.000000 &  0.000000  \\ 
 1 &  1.028824 &  0.173903 &  1.73e-01 &  0.000000 &  0.000000  \\ 
 2 &  1.043876 &  0.346136 &  4.52e-04 &  2.901361 &  0.072991  \\ 
 3 &  1.043988 &  0.346574 &  1.00e-11 &  2.962158 &  0.081198  \\ 
 4 &  1.043988 &  0.346574 &  1.58e-34 &  2.979257 &  0.092628  \\ 
 5 &  1.043988 &  0.346574 &  6.56e-103 &  2.998820 &  0.152023  \\ 
 6 &  1.043988 &  0.346574 &  4.70e-308 &  2.999999 &  0.166638  \\ 
 		\end{tabular}
	\caption{Calculations (i) of Example \ref{Ex2}.}
	\label{tab:overall_example2}
\end{table}

\begin{figure}[H]
\centering
\includegraphics[width=0.55\textwidth]{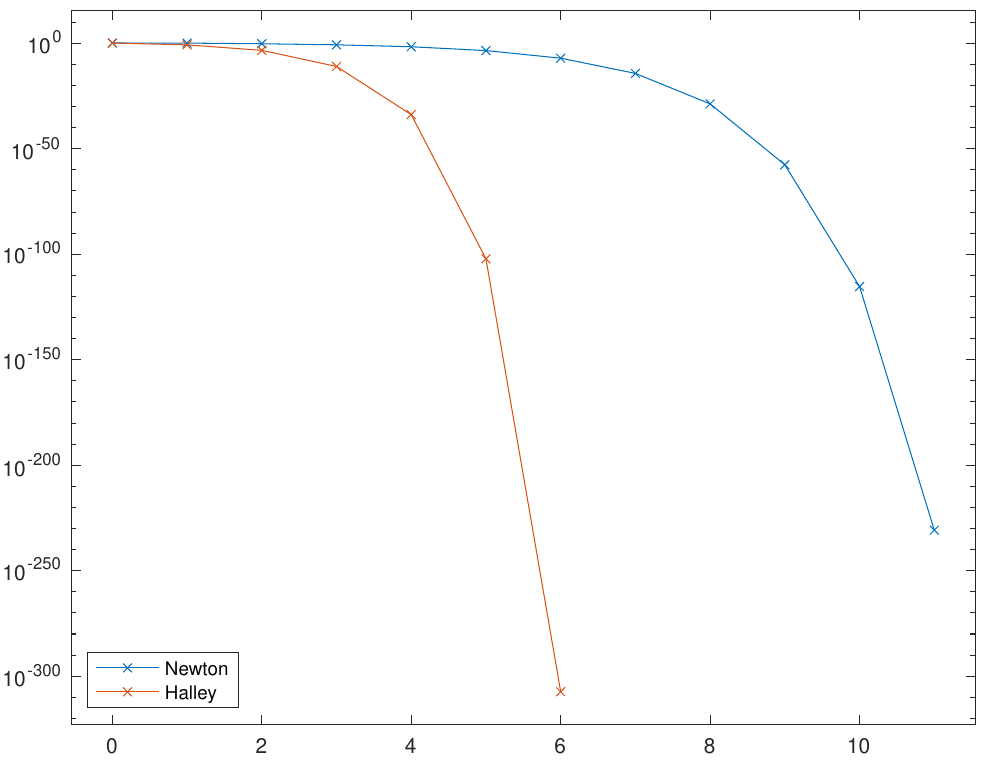}
\caption{Halley vs. Newton method in case {\rm (i)} of Example \ref{Ex2}.  Comparison of convergence of $e_k$ versus the $k$-th iteration.  }
	\label{fig:comparison_example2}
\vspace{0.5cm}
\end{figure}


To compare the total computational cost of the (Josephy–)Halley and (Josephy–)Newton methods, we apply both algorithms to a set of initial guesses
$x_0 \in \mathbb{R}^2$
sampled on a uniform grid over \([-4,4]\times[-4,4]\), as shown in Figure~\ref{fig:comparison_example2_grid} and Figure~\ref{fig:comparison_example2_grid2}. The left panel classifies each grid point into three cases:
\begin{enumerate}
  \item[Case 0] both methods converge, with (Josephy–)Newton achieving a lower total cost (dark blue);
  \item[Case 1] both methods converge, with (Josephy–)Halley achieving a lower total cost (teal);
  \item[Case 2] (Josephy–)Halley converges within the 200-iteration cap, whereas Josephy–Newton fails to converge (reaching the 200-iteration limit) (yellow).
\end{enumerate}
The middle panel shows number of iterations required by the (Josephy–)Newton method (capped at its 200-iteration limit), and the right panel shows number of iterations required by the (Josephy–)Halley method.

Let us emphasize that, to guarantee a fair comparison of the two schemes, we solve both the inclusion in \eqref{eqIterationJosephyNewton} and the pair of inclusions in \eqref{eqIterationHalley} by calling the exact same numerical subroutine. Specifically, each inclusion is passed to a robust interval‐Newton–like solver which uses identical stopping criteria ($1e-300$).

By holding the implementation details, convergence thresholds, and underlying arithmetic operations constant, any observed differences in runtime or iteration counts can be attributed solely to the structural distinctions between the Josephy–Newton and Josephy–Halley updates, rather than to variations in the solution machinery itself.

\begin{figure}[H]
\includegraphics[width=1\textwidth]{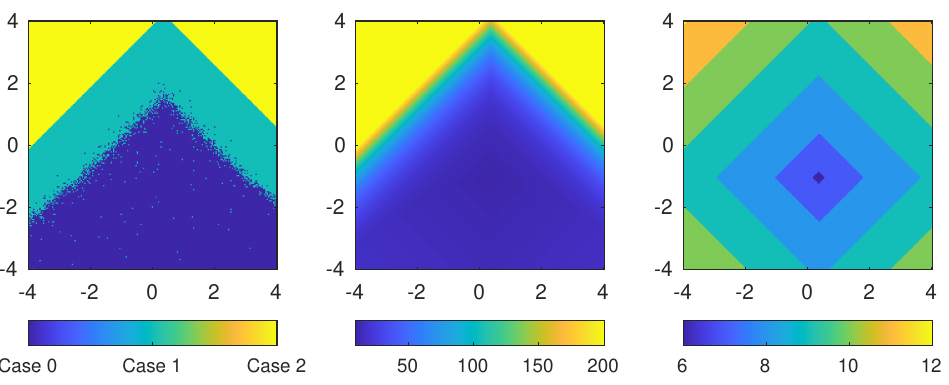}
\caption{Comparison of convergence: Halley vs. Newton of Example \ref{Ex2} (i).}
	\label{fig:comparison_example2_grid}
\end{figure}

\begin{figure}[H]
\includegraphics[width=1\textwidth]{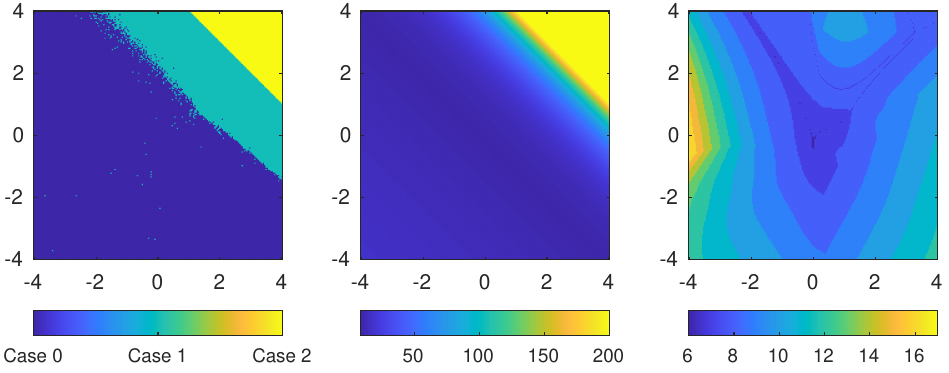}
\caption{Comparison of convergence: Josephy-Halley vs. Josephy-Newton of Example \ref{Ex2} (ii).}
	\label{fig:comparison_example2_grid2}
\end{figure}
\end{example1}

\section{Conclusion}
We have presented a novel Josephy–Halley iteration for generalized equations that achieves cubic convergence under mild smoothness and metric regularity assumptions.  Our local analysis (Theorem~3.1) shows that, whenever the single-valued part has a Hölder-continuous second derivative and the linearization is metrically regular, the method converges with order \(2+p\).  The semilocal Kantorovich-style result (Theorem~4.1) provides explicit, computable bounds on the convergence region and error estimates via a majorant function.  Implementation in MATLAB (with variable-precision arithmetic) on benchmark problems demonstrates the practical advantages of incorporating second-order information: the Josephy–Halley scheme requires significantly fewer iterations than the Josephy–Newton predictor, while retaining comparable overall cost.  Future work may explore extensions to nonsmooth single-valued mappings, inexact subproblem solvers, and applications to large-scale variational inequalities arising in optimization and equilibrium modeling.



\bibliographystyle{acm}

\end{document}